\documentclass[12pt, pdfmode]{amsart}
\usepackage{amsmath}
\usepackage{amssymb}
\usepackage{amscd} 
\usepackage{graphicx}

\newcommand{\bz}{\mathbf Z}

\newcommand{\bc}{\mathbf C}

\newcommand{\ind}{\operatorname{ind}}

\newcommand{\spin}{\text{spin}}

\newcommand{\sgn}{\operatorname{sgn}}

\newcommand{\Pin}{\operatorname{Pin}}

\newtheorem{theorem}{Theorem}
\newtheorem{proposition}{Proposition}
\newtheorem{corollary}{Corollary}

\theoremstyle{definition}
\newtheorem{definition}{Definition}
\newtheorem{remark}{Remark}

\begin{document}
\title[ ]
{On Manolescu's $\kappa$-invariants of rational homology $3$-spheres 
}

\author[]{Masaaki Ue}

\address{
Kyoto University, Kyoto, 606-8502, Japan}
\email{ue.masaaki.87m@st.kyoto-u.ac.jp}

%\subjclass{57R57, 57M27, 57N10}
%\keywords{}
\begin{abstract}
We give an estimate for Manolescu's $\kappa$-invariant of a rational homology $3$-sphere $Y$ by the data of a spin $4$-orbifold bounded by $Y$.
By an appropriate choice of a $4$-orbifold, sometimes we can restrict and determine the value of $\kappa$.  
we give such examples in case of rational homology $3$-spheres obtained by Dehn surgeries along knots in $S^3$. 
\end{abstract}

\maketitle

Let $Y$ be a rational homology 3-sphere with spin structure $\mathfrak s$. Then there exists a 
compact smooth 4-manifold  with $\spin$ structure $(X,\mathfrak s_X )$ that satisfies $\partial (X , \mathfrak s_X ) = (Y, \mathfrak s )$. 
Manolescu's $\kappa$-invariant $\kappa (Y, \mathfrak s )$
(which is simply denoted by $\kappa (Y)$ if $Y$ is an integral homology 3-sphere), defined via a Seiberg-Witten Floer homotopy type of $(Y, \mathfrak s )$ \cite{Mano1}, \cite{Mano3}, 
gives constraints on the intersection form of $X$ \cite{Mano4}. In case of almost rational homology 3-spheres, Dai-Sasahira-Stoffregen \cite{DSS} determines their 
Seiberg-Witten Floer homotopy types and the values of $\kappa$-invariants. 
But it is still difficult to 
determine $\kappa (Y, \mathfrak s )$ in general. In this paper we 
give an estimate of $\kappa (Y, \mathfrak s )$ by the data of a compact $\spin$ $4$-orbifold $X$ bounded by $Y$. If $b_2 (X)$ is small enough, 
the value of $\kappa (Y, \mathfrak s )$ is restricted and sometimes determined. In \S 2 we apply it to Dehn surgeries along 
knots in $S^3$. 
For example, let $K_{1/n}$ be a $(1/n)$-surgery along a knot $K$. Then 
$\kappa (K_{1/n}) =0$ for any knot $K$ if $n$ is even and positive (Theorem \ref{Dehn}), while if $n$ is odd and positive, every 
integer is realized as the value of $\kappa (K_{1/n} )$ for a certain knot $K$ (an explicit example is given in Proposition \ref{twist}). 

\section{An estimate of $\kappa (Y, \mathfrak s )$ by a $\spin$ 4-orbifold bounded by $Y$} 
Let $(Y, \mathfrak s )$ be a rational homology 3-sphere with $\spin$ structure. Then $Y$ bounds a compact connected 4-orbifold $X$ with $\spin$ structure 
$\mathfrak s_X$, which induces $\mathfrak s$ on $Y$. 

\begin{theorem}\label{estimate}
Let $(Y, \mathfrak s )$ be a rational homology 3-sphere with $\spin$ structure, and $( X, \mathfrak s_X )$ be a compact connected $\spin$ 4-orbifold  bounded 
by $(Y, \mathfrak s )$. We assume that  the set of the singular points of $X$ consists of isolated points only, which is denoted by 
$\{ x_i  \ | \ i=1 ,\dots , n \}$. Then a regular neighborhood of $x_i$ is a cone $cS_i$ over some spherical 3-manifold $S_i$. 
Let $\mathfrak s_i$ be the $\spin$ structure on $S_i$ induced by $\mathfrak s_X$, and $k^{\pm} = b_2^{\pm} (X)$. Then $\kappa (Y, \mathfrak s )$ has the following properties. 

\begin{enumerate}
\item 
\[
\kappa (Y,  \mathfrak s ) \equiv -\frac{k^+ -k^- }{8} -\sum_{i=1}^n \overline{\mu } (S_i , \mathfrak s_i ) \pmod{2\bz}.
\]
\item 
\begin{align*} 
-k^+ -\frac{k^+ -k^-}{8} -\epsilon^+  &-\sum_{i=1}^n \overline{\mu} (S_i , \mathfrak s_i ) \le \kappa (Y, \mathfrak s ) \\
&\le 
k^- -\frac{k^+ -k^- }{8} +\epsilon^- -\sum_{i=1}^n  \overline{\mu} (S_i , \mathfrak s_i ),  
\end{align*}
where
\[
\epsilon^- = \begin{cases} 0 & (k^- \equiv 0 \pmod 2 ) \\
 1 & (k^- \equiv 1 \pmod 2 ) 
 \end{cases} , \ 
 \epsilon^+ = 
 \begin{cases} 
 0 & (k^+ =0 ) \\ -1 & (k^+ \equiv 1 \pmod 2 ) \\
 -2 & (k^+ \equiv 0 \pmod 2 , k>0 ) 
 \end{cases} 
 \]
 and $\overline{\mu} (S_i , \mathfrak s_i )$ is the Neumann-Siebenmann invariant of $(S_i , \mathfrak s _i )$ \cite{N}, \cite{Si}. 
 \end{enumerate}
 
 \end{theorem}
 
 The proof of the following corollary is straightforward from Theorem \ref{estimate}.
 
 \begin{corollary}\label{cor}
 Let $\kappa_0 = -\dfrac{k^+ - k^- }{8} -\sum\limits_{i=1}^n \overline{\mu} (S_i , \mathfrak s_i )$. Then 
 \[
 \kappa (Y, \mathfrak s ) = \begin{cases} 
 \kappa_0 &\text{if $k^+ \le 2$ and $k^- =0$}, \\
 \text{$\kappa_0$ or $\kappa_0 -2$} &\text{if $k^+ =3,4$ and $k^- =0$}, \\
 \text{$\kappa_0$ or $\kappa_0 +2$} &\text{if $k^+ \le 2$ and $k^- =1,2$}.
 \end{cases} 
 \]
 \end{corollary}
 
 The proof of Theorem \ref{estimate} is based on the following inequalities by Manolescu. (In \cite{U4} we give the outline of the proof in case of a rational homology 3-sphere.) 
 
 \begin{theorem}\label{ineq}\cite{Mano4} 

Let $(Y_i , \mathfrak s_i )$ $(i=0,1 )$ be rational homology 3-spheres with $\spin$ structures and 
$(W , \mathfrak s_W )$ be a $\spin$ cobordism with $b_1 (W) =0$ from $(Y_0 , \mathfrak s_0 )$ to $(Y_1 , \mathfrak s_1 )$. 
Then 
\[
-b_2^- (W) +\frac{\sigma (W)}{8} -1 \le \kappa (Y_0, \mathfrak s_0 ) -\kappa (Y_1 , \mathfrak s_1 ) \le b_2^+ (W) +\frac{\sigma (W)}{8} +1 . 
\]
The first term of the above inequality can be replaced by $-b_2^- (W) +\frac{\sigma (W)}{8}$ if $b_2^- (W)$ is even, and 
the last term can be replaced by $b_2^+ (W) +\frac{\sigma (W)}{8}$ if $b_2^+ (W)$ is even. 

Furthermore if $(Y_0 , \mathfrak s_0 )$ is Floer $K_G$ split
{\footnote{See \cite{Mano4} for the definition of this terminology.}}, then the second inequality is improved as follows. 
\begin{align*}
&\kappa (Y_0, \mathfrak s_0 ) -\kappa  (Y_1 , \mathfrak s_1 )   \\
&\le \begin{cases} 
b_2^+ (W) +\frac{\sigma (W)}{8} -1 &\quad \text{if $b_2^+ (W)$ is odd}, \\
b_2^+ (W) +\frac{\sigma (W)}{8} -2 & \quad \text{if $b_2^+ (W)$ is even and $b_2^+ (W) >0$. } 
\end{cases}
\end{align*}
\end{theorem}

\subsection*{Proof of Theorem \ref{estimate}}  \hfill 
\par
\medskip
Let $cS_i$ $(i=1, \dots , n )$ be the cones over the spherical 3-manifolds $S_i$, which are the regular neighborhoods of the singular points of $X$. 
If $n\ge 2$, $S_i$'s can be connected by 1-handles in $X \setminus \mathrm{Int} \cup_{i=1}^n cS_i $ to form a $\spin$ cobordism 
$W \subset X \setminus \mathrm{Int} \cup _{i=1}^n cS_i$ from $\cup_{i=1}^n S_i$ to the connected sums $S_0 =\sharp_{i=1}^n S_i$ of $S_i$'s. 
Then $X_0 =X \setminus \mathrm{Int} (\cup_{i=1}^n cS_i  \cup W )$ is a $\spin$ cobordism from $S_0$ to $Y$. 
Let $\mathfrak s_0$, $\mathfrak s_W$, $\mathfrak s_{X_0}$ be the $\spin$ structures on $S_0$, $W$, and $X_0$ induced by $\mathfrak s_X$ respectively. 

Since $S_i $ $(i\ge 0 )$ has a metric $g_i$ of positive scalar curvature, $\kappa (S_i , \mathfrak s_i )$ is related to the correction term $n (S_i , \mathfrak s_i , g_i )$ of the Seiberg-Witten-Floer homotopy type of 
$(S_i, \mathfrak s_i )$ with respect to the metric $g_i$ by the following proposition. 

\begin{proposition}\label{psc} \cite{Mano4}
Let $(Y, \mathfrak s)$ be a rational homology 3-sphere with $\spin$ structure that admits a metric $g$ of positive scalar curvature. Then 
$\kappa (Y, \mathfrak s ) = -n( Y, \mathfrak s , g )$ and $(Y, \mathfrak s )$ is Floer $K_G$ split. 
\end{proposition}

If $i\ge 1$, 
$\kappa (S_i , \mathfrak s_i )$ is also related to the Fukumoto-Furuta invariant 
$w(S_i, cS_i , \mathfrak s_{cS_i} )$, where $\mathfrak s_{cS_i}$ is the restriction of $\mathfrak s_X$ to $cS_i$. 

 \begin{definition}\label{FF}
 Let $(X, \mathfrak s_X )$ be a 
 compact $\spin$ 4-orbifold with $\spin$ structure $\mathfrak s_X $ bounded by
 a rational homology 3-sphere $(Y, \mathfrak s )$ with $\spin$ structure. 
 Choose a compact spin 4-manifold $(X' , \mathfrak s_{X'} )$ 
 with $\partial (X' , \mathfrak s_{X'}  ) =(Y, \mathfrak s )$ and put $(Z, \mathfrak s_Z ) = (X\cup (-X') , \mathfrak s_X \cup 
 \mathfrak s_{X '} )$. Then the Fukumoto-Furuta invariant $w(Y, X, \mathfrak s_X )$  is defined to be 
 \[
 w(Y, X, \mathfrak s_X ) =-\ind_{\bc} \mathcal D_Z  (\mathfrak s_Z ) +\frac{\sigma (X' )}{8}, 
 \]
 where $\mathcal D_Z (\mathfrak s_Z )$ is the Dirac operator of $(Z, \mathfrak s_Z )$. 
 The value of $w(Y, X, \mathfrak s_X )$ does not depend on the choice of $X'$. 
 (The sign of $w$ is opposite to the original one in \cite{FF} .) 
\end{definition}

Since $\dim \ker\mathcal D_{S_i} (\mathfrak s_i ) =0$, $n ( S_i , \mathfrak s_i , g_i )$ $(i\ge 1 )$ is represented by the eta invariants of the Dirac and the signature operators 
 on $S_i $ as follows:
 \[
  n(S_i , \mathfrak s_i  , g_i ) = -\frac 18 (4\eta^{\mathrm{Dir}} (S_i , \mathfrak s_i  , g_i  ) +\eta^{\mathrm{sign}} (S_i , g_i  )). 
 \] 
  Here we assume that the Clifford multiplication $c$ of the volume form $\mathrm{vol}_{S_i}$ on $S_i$ satisfies 
   $c(\mathrm{vol}_{S_i} )=-1$ as in \cite{U4}. If we choose $c$ satisfying $c(\mathrm{vol} _{S_i}) =1$, $\eta^{\mathrm{Dir}} (S_i , \mathfrak s_i , g_i )$ 
   should be replaced with $-\eta^{\mathrm{Dir}} (S_i , \mathfrak s_i , g_i )$ (\cite{U4}, Remark 4).

 The right hand side of the above formula coincides with $w(S_i , sS_i , \mathfrak s_{cS_i } )$, which is the same as 
   $\overline{\mu} (S_i , \mathfrak s_i )$ (\cite{U4}), and hence 
   \[
  (*) \quad  \kappa (S_i , \mathfrak s_i ) =-n(S_i , \mathfrak s_i , g_i ) =-\overline{\mu} (S_i , \mathfrak s_i ). 
   \]
   By considering the $\spin$ cobordism $W$ from $\cup_{i=1}^n S_i$ to $S_0$, we also have the relation 
   $n(S_0 , \mathfrak s _0 , g_0 ) =\sum_{i=1}^n n (S_i , \mathfrak s_i , g_i )$ (\cite{U4}). It follows that 
   \[
   (**) \quad 
   \kappa (S_0 , \mathfrak s_0 ) = -n(S_0 , \mathfrak s_0 , g_0 ) =- \sum_{i=1}^n n (S_i , \mathfrak s_i  , g_i ) =
   -\sum_{i=1}^n \overline{\mu} (S_i , \mathfrak s_i ).
   \]
   
   Then applying the inequality in Theorem \ref{ineq} (note that $(S_0 , \mathfrak s_0 )$ is 
Floer $K_G$ split) to the $\spin$ cobordism $X_0$ and $(**)$, we obtain the inequality in 
   Theorem \ref{estimate} (2). We note that for a rational homology 3-sphere with $\spin$ structure $(Y, \mathfrak s )$, 
   the mod $2\bz$ reduction of $\overline{\mu} (Y ,\mathfrak s )$ is the Rockhlin invariant $\mu (Y, \mathfrak s )$ of $(Y, \mathfrak s )$. 
   We also have $\kappa (Y, \mathfrak s ) \equiv -n(Y, \mathfrak s , g ) \equiv -\mu (Y, \mathfrak s ) \mod{2\bz}$ for some metric $g$. 
   Thus by the additivity of signature, if we choose a $\spin$ 4-manifold $(W_i , \mathfrak s_{W_i} )$ with 
   $\partial (W_i , \mathfrak s_{W_i}  ) =(S_i , \mathfrak s_i )$, we have 
\[
\kappa (Y, \mathfrak s ) \equiv -\overline{\mu} (Y, \mathfrak s ) \equiv -\frac 18 (\sigma (X_0 ) +\sum_{i=1}^n \sigma (W_i ) ) 
\equiv  -\frac 18 \sigma (X_0 ) -\sum_{i=1}^n \overline{\mu} (S_i , \mathfrak s_i  ) \mod{2\bz}. 
\]
This proves Theorem \ref{estimate} (1).  \qed
   
\medskip   
 On the other hand, $w(Y, X, \mathfrak s_X )$ for $(Y, X )$ in Theorem \ref{estimate} is represented via the eta invariants of $S_i$'s and $(*)$ as follows (\cite{U4}). 

\[
 w(Y, X, \mathfrak s_X ) = \frac{\sigma (X)}{8} +\sum_{i=1}^n \overline{\mu} (S_i , \mathfrak s_i ) .
\]
Furthermore $w(Y,X, \mathfrak s_X ) \equiv \mu (Y , \mathfrak s ) \mod {2\bz}$ since $\dim \ind_{\bc} \mathcal D_Z (\mathfrak s_Z )$ is even, where 
$Z$ is a closed $\spin$ orbifold in the definition of $w(Y, X, \mathfrak s_X )$. 
The following theorem is deduced from these facts and Theorem \ref{estimate}. 

\begin{theorem}\label{compare}
Let $(Y, \mathfrak s)$ be a rational homology 3-sphere with $\spin$ structure and $(X, \mathfrak s_X )$ be a $\spin$ 4-orbifold bounded by $(Y, \mathfrak s )$ 
as in Theorem 1. Then 
\begin{align*}
&\kappa (Y, \mathfrak s ) +w(Y, X, \mathfrak s_X ) \equiv 0 \mod {2\bz}, \\
-k^+ -&\epsilon^+ \le \kappa (Y, \mathfrak s ) +w(Y, X, \mathfrak s_X ) \le k^- +\epsilon^-  .
\end{align*}
\end{theorem}

If $Y$ is a Seifert rational homology 3-sphere, we can choose $X$ so that $w(Y, X , \mathfrak s_X )=\overline{\mu} (Y, \mathfrak s )$ and 
$k^{\pm}$ are small, which restricts the value of $\kappa (Y, \mathfrak s ) + \overline{\mu} (Y, \mathfrak s )$ (\cite{U4}). 

\medskip
We also note that the $\kappa$-invariant of an almost rational (AR) homology 3-sphere is determined by Dai-Sasahira-Stoffregen \cite{DSS}. 

\begin{theorem} \label{AR}\cite{DSS} 
Let $(Y, \mathfrak s )$ be an AR homology 3-sphere with $\spin$ structure. Then either one of the followings holds. 
\begin{enumerate}
\item If $-\overline{\mu} (Y, \mathfrak s ) =\delta (Y, \mathfrak s)$, then $\kappa (Y, \mathfrak s ) = -\overline{\mu} (Y, \mathfrak s )$. 
\item 
If $-\overline{\mu} (Y, \mathfrak s ) < \delta (Y, \mathfrak s )$, then $\kappa (Y, \mathfrak s ) = -\overline{\mu} (Y, \mathfrak s ) +2$. 
\end{enumerate}
Here $\delta (Y, \mathfrak s )$ is the monopole Fr{\o}yshov invariant of $(Y, \mathfrak s )$. Note that we have always 
$-\overline{\mu} (Y, \mathfrak s ) \le \delta (Y, \mathfrak s )$ \cite{Dai}. 
\end{theorem}

\begin{remark}\label{rem1}
 In \cite{F}, \cite{KM}, the Fr{\o}yshov invariant is defined as $h(Y, \mathfrak s ) = -\delta (Y, \mathfrak s )$ so that 
 $d(Y, \mathfrak s ) = -2h(Y, \mathfrak s )$, where $d(Y, \mathfrak s )$ is the Ozsv\'ath-Szab\'o's correction term \cite{OS}. 
 \end{remark}

 For an AR homology 3-sphere $Y$, the correction term $\underline{d} (Y, \mathfrak s )$ of the involutive Floer homology is equal to $-2\overline{\mu} (Y ,
 \mathfrak s )$ \cite{DM}. If $Y$ is also an $L$-space, then the correction terms $\underline{d} (Y, \mathfrak s )$, $d(Y, \mathfrak s )$, 
 $\overline{d} (Y, \mathfrak s )$ are all the same (\cite{HM}). Since $\delta (Y, \mathfrak s ) = d(Y, \mathfrak s )/2$, we have 
 $-\overline{\mu} (Y, \mathfrak s ) =\delta (Y, \mathfrak s) = \kappa (Y, \mathfrak s )$ by the above theorem. $(*)$ is also deduced from this fact because 
 a spherical 3-manifold is rational and also an $L$-space (\cite{Ne1}). 
 
\section{$\kappa$-invariants of Dehn surgeries along knots in $S^3$} 
 
 We apply Theorem \ref{estimate} to a rational homology 3-sphere obtained by Dehn surgery along a knot $K$ in $S^3$. 
 We denote by $K_{p/q}$ the $(p/q)$-surgery along $K$. 
 The set of $\spin$ structures on  the complement $X_K : = S^3 \setminus \mathrm{Int} N(K)$ of  the tubular neighborhood $N(K)$ of $K$ is 
 identified with $\mathrm{Hom}_{\bz_2}  (H_1 (X_K , \bz_2 ) , \bz_2 ) \cong \bz_2$. 
 Hence there are 2 $\spin$ structures $\mathfrak s_0$ and $\mathfrak s_1$ on $X_K$. Here $\mathfrak s_0$ is the one that extends to the unique $\spin$ structure on $S^3$, and $\mathfrak s_1$ is the one that does not extend to $S^3$.  
The $\spin$ structure on $X_K$ corresponding to $c\in \mathrm{Hom}_{\bz_2}  (H_1 (X_K , \bz_2 ) , \bz_2 )$ extends to $K_{p/q}$ if and only if 
$pc (\mu) \equiv pq \mod 2$, where $\mu$ is the meridian of $K$. It follows that only $\mathfrak s_0$ 
(resp. $\mathfrak s_1 $) extends to 
 the $\spin$ structure on $K_{p/q}$ if 
 $p$ is odd and $q$ is even (resp. both $p$ and $q$ are odd), while both $\mathfrak s_0$ and $\mathfrak s_1$ extend to $K_{p/q}$ 
 if $p$ is even and $q$ is odd. We denote the extended $\spin$ structures on $K_{p/q}$ by the same symbols as those on $X_K$. 
 
 \begin{theorem}\label{Dehn} 
 Let $K$ be any knot in $S^3$, and $p$ and $q$ are positive integers such that $\gcd (p,q)=1$, $p+q \equiv 1 \mod 2$. 
 Then 
 \begin{align*} 
&\kappa (K_{p/q} , \mathfrak s_0 ) =\kappa (L ( p,-q ) , \mathfrak s_0  ) , \\ 
& \kappa (K_{-p/q}  , \mathfrak s_0 ) = \text{$\kappa (L ( p , q ) , \mathfrak s_0 )$ or $\kappa (L(p, q ) , \mathfrak s_0 ) +2$} . 
\end{align*} 
Here $\mathfrak s_0$ on $L ( p, \pm q )$ is the $\spin$ structure $\mathfrak s_0$ on the $\mp p/q$-surgery along the unknot, which is 
$L(p, \pm q )$. We note that $\kappa (L(p, \pm q ) ,\mathfrak s_0 ) =-\overline{\mu} (L(p, \pm q ) , \mathfrak s_0 )$. 
 In particular, we have 
 \[
 \kappa (K_{1/2n} , \mathfrak s_0 ) = 
 \begin{cases} 
 0 & \text{if $n>0$}, \\ 
 \text{$0$ or $2$} & \text{if $n<0 $}. 
 \end{cases} 
 \]
\end{theorem}

\begin{proof} 
First suppose that $p>q >1$. Since $p+q \equiv 1 \mod 2$, we have the 
continued fraction expansion of $p/q$ of the following form. 

\[
 (\dagger ) \quad 
\frac{p}{q} =[\alpha_1 , \alpha_2 , \dots , \alpha_n ] =
\alpha_1 -\cfrac{1}{\alpha_2 -\cfrac{1}{\cdots -\cfrac{1}{\alpha_n}}}
\]
where $n\ge 2$ and $\alpha_i \equiv 0 \mod 2$ and $|\alpha_i | \ge 2$ for all $\alpha_i$. 
Then $K_{p/q}$ bounds a 4-manifold $X$ represented by the framed link $\mathcal L$ in Figure 1. 

\begin{figure}
\begin{center}
\includegraphics[width=7cm, clip]{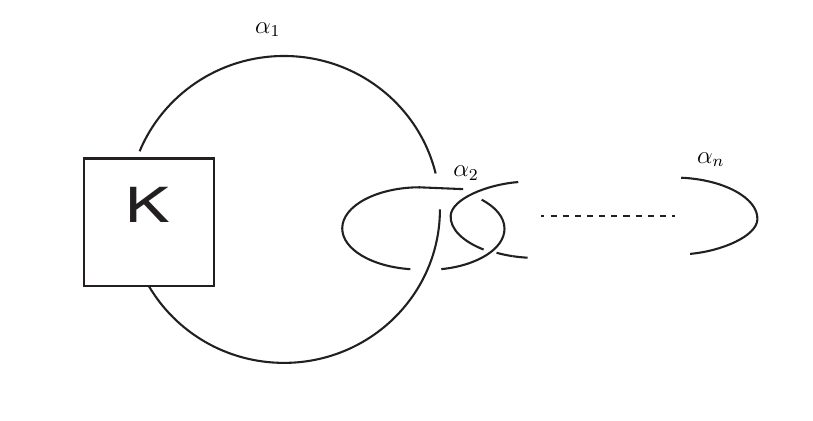}
\caption{The framed link $\mathcal L$}
\end{center}
\end{figure}

Since every $\alpha_i$ is even, $X$ has a $\spin$ structure $\mathfrak s_X$ whose restriction to $K_{p/q}$ is 
$\mathfrak s_0$ (see Remark \ref{rem2} (1) below). The sublink $\mathcal L_0$ of $\mathcal L$ consisting of the components of framing 
$\alpha_i$ $(i\ge 2 )$ represents a plumbing $X_{\Gamma}$ whose boundary is a lens space $L(p' , -q' )$, where 
$p' / q' = [\alpha_2 , \cdots , \alpha_n ]$. We can push $X_{\Gamma}$ into $\mathrm{Int} X$ and attach the cone 
$cL ( p ' , -q ' )$ over $L(p' , -q' )$ to the complement $X_0 =X\setminus \mathrm{Int} X_{\Gamma}$ along $L ( p ' , -q' )$ to form a 4-orbifold 
$\widehat X$. The $\spin$ structure $\mathfrak s_X$ over $X_0$ can be extended to $cL(p' , -q ' )$ and form the orbifold $\spin$ structure 
$\mathfrak s_{\widehat X}$ on $\widehat X$ (This procedure is similar to those in \cite{S}, \cite{U2}). The intersection matrices of $X$ and $X_{\Gamma}$ are 
congruent to the diagonal matrices whose diagonal entries are 
$\{ [\alpha_i, \alpha_{i-1} , \cdots , \alpha_n ]  \ | \  i \ge 1 \}$ and $\{ [\alpha_i , \alpha_{i-1} , \cdots , \alpha_n ] \ | \ i\ge 2 \}$ respectively. 
Hence $\sigma (X_0 ) =\sigma (X) -\sigma (X_{\Gamma} ) $ is the sign of $[\alpha_1 , \cdots , \alpha_n ] =p/q $, which is 1. 
Since $b_2 ( X_0 ) =b_2 (X) -b_2 (X_{\Gamma} ) =1$, we have 
$b^+_2 (\widehat X) = b_2^+ (X_0 ) =1$ and $b^-_2 (\widehat X) = b_2^- (X_0 ) =0$ and hence by Corollary \ref{cor}, 
\[
 \kappa (K_{p/q}, \mathfrak s_0 ) = -\frac 18 -\overline{\mu} (L ( p' , -q ' ) , \mathfrak s '  ) ,
\]
where $\mathfrak s '$ is the restriction of $\mathfrak s_X $ to $L (p' , -q ' )$. Since $\sgn [\alpha_i , \cdots , \alpha_n ] =\sgn \alpha_i $ and 
$\sgn \alpha_1 =1$ and the characteristic sublink of $\mathcal L_0$ associated with $\mathfrak s'$ is empty, 
we have $\overline{\mu} (L ( p ' , -q' ) , \mathfrak s' )=\frac 18 \sigma (X_{\Gamma} )$ and 
\[
\kappa (K_{p/q} , \mathfrak s_0 ) = -\frac 18 (1+\sigma (X_{\Gamma} ) ) =-\frac 18 \sum_{i=1}^n \sgn \alpha_i ,
 \]
 which is the same as $-\overline{\mu} (L ( p ,-q ) , \mathfrak s_0 ) = \kappa (L( p, -q ) , \mathfrak s_0 )$. In fact, $\kappa (K_{p/q} , \mathfrak s_0 )$ 
 does not depend on $K$, and hence is the same as the $\kappa$-invariant of the $p/q$-surgery along the unknot, which is $L(p, -q )$. 
 In case of $K_{-p/q}$, we form a $\spin$ 4-manifold $X'$ with $\partial X' = K_{-p/q}$ represented by a framed link $\mathcal L '$, which is obtained from 
$\mathcal L$ by replacing the framings $\alpha_i$ with $-\alpha_i$. There is a plumbing $-X_{\Gamma}$ corresponding to the sublink of $\mathcal L '$, which 
consists of the components  
of framing $-\alpha_i$ $(i\ge 2 )$ with 
$\partial (-X_{\Gamma} ) =L (p ' , q ' )$ in $X'$. Then by replacing it with the cone over $L ( p ' , q ' )$, we also obtain the $\spin$ 4-orbifold $\widehat X '$. 
 In this case $b^+_2 (\widehat X ' ) =0$ and $b^- (\widehat X ' ) =1$ and hence by Corollary \ref{cor},  
 \[
 \kappa (K_{-p/q}  , \mathfrak s_0 ) =\frac 18 -\overline{\mu} (L ( p ' , q ' ) , \mathfrak s' ) \ \text{or} \ 
 2+\frac 18 -\overline{\mu} (L(p' , q ' ) , \mathfrak s ' ) ,
 \]
 where $\mathfrak s'$ is the restriction of the $\spin$ structure on $X ' $ to $L ( p' , q ' )$. Here we have 
\[
\frac 18 -\overline{\mu} (L (p' , q ' ) , \mathfrak s ' ) =\frac 18 \sum_{i=1}^n \sgn \alpha_i =-\frac 18 \overline{\mu} (L ( p , q ), \mathfrak s_0 ) 
= \kappa (L ( p ,q ) , \mathfrak s_0 ).
\]
 
 If $p$ is even and $q=1$, $K_p$ bounds a $\spin$ 4-manifold $X$ represented by a framed link $K$ with framing $p$ whose $\spin$ structure $\mathfrak s_X$ 
 induces $\mathfrak s_0$ on $K_p$. In this case we apply Manolescu inequality in Theorem 2 (by putting $Y_0 =S^3 $) directly to $(X, \mathfrak s_X )$. Then 
 $\kappa (K_p , \mathfrak s_0 ) =-1/8  = -\overline{\mu} (L ( p, -1 ) , \mathfrak s_0 ) = \kappa (L ( p, -1 ) , \mathfrak s_0 )$, 
 and $\kappa (K_{-p} , \mathfrak s_0 )$ is either $\kappa (L ( p , 1 ) , \mathfrak s_0 )$ or $\kappa (L(p , 1 ), \mathfrak s_0 ) +2$ since 
 $\kappa (Y_{\pm p}  , \mathfrak s_0 ) \equiv -{\mu} (Y_{\pm p} , \mathfrak s_0 )  \equiv \mp 1/8 \pmod{2\bz}$. Thus we prove 
Theorem \ref{Dehn} in the first case. 

Next suppose that $q>p>0$. Then there is a continued fraction expansion of $q/p$ of the form 
\[
\frac qp =[\beta_2 , \dots , \beta_n ] \quad n\ge 2 , \ |\beta_i |\ge 2 , \ \beta_i  \equiv 0\mod 2  \ \text{for all $i \ge 2$}.
\]
In this case $K_{p/q}$ bounds a $\spin$ 4-manifold $X ' $ represented by a framed link $\mathcal L '$, which is obtained from the above 
$\mathcal L$ by replacing the framing $\alpha_1$ with $0$, and $\alpha_i $ with $-\beta_i$ for $i\ge 2$. The $\spin$ structure $\mathfrak s_{X ' }$ 
on $X ' $ induces $\mathfrak s_0$ on $K_{p/q}$. Then we have a plumbing $X_{\Gamma} '$ in $X ' $ corresponding to the sublink, which consists of the components of framing 
$-\beta_i$ $ (i\ge  2 )$. The boundary of $X_{\Gamma} ' $ is a lens space $L ( p ' , q ' ) $, where $p' /q ' = 
[\beta_2 , \cdots , \beta_n ]$. Again we can push $X_{\Gamma} ' $ into $\mathrm{Int} X ' $ and replace it with the cone 
$cL( p ' , q ' ) $ over $L ( p ' , q ' )$ to form a $\spin$ 4-orbifold $\widehat X '$. The intersection matrices of $X '$ and $X_{\Gamma} '$ are 
congruent to the diagonal matrices whose diagonal entries are 
$\{ 1/[\beta_2 , \cdots , \beta_n ] , -[\beta_2 , \cdots , \beta_n ] , \cdots -\beta_n \}$ and 
$\{ - [ \beta_2 , \cdots , \beta_n ] , \cdots , -\beta_n \}$ respectively. Hence $\sigma (\widehat X ' )$ is the sign of 
$1/[\beta_2 , \cdots , \beta_n ] =p/q$ and hence $1$. It follows that $b_2^+ (\widehat X ' ) =1$ and $b_2^- (\widehat X ' ) =0$. 
 Thus again by Corollary \ref{cor} we have 
 \begin{align*}
& \kappa (K_{p/q} , \mathfrak s_0 ) =-\frac 18 -\overline{\mu} (L ( p ' ,q ' ) , \mathfrak s ' )  = 
 -\frac 18 ( 1-\sum_{i=2}^n \sgn [\beta_i , \cdots , \beta_n ] )  \\
 &=
 -\frac 18 (\sgn (1/[\beta_2 , \cdots , \beta_n ] + \sum_{i=2}^n  \sgn [-\beta_i , \cdots , -\beta_n ]) \ (\mathfrak s ' = \mathfrak s_{X '} |_{L ( p ' , q ' )}) ,
 \end{align*}
which is the same as $-\overline{\mu} (L ( p,-q ) , \mathfrak s_0 ) = \kappa (L(p,- q ), \mathfrak s_0 )$. 
 In case of $K_{-p/q}$, it bounds a $\spin$ 4-manifold $X ''$ represented by a framed link $\mathcal  L ''$, which is obtained 
 from $\mathcal L ' $ by replacing the framing $-\beta_i$ with $\beta_i $ $(i\ge 2 )$. The $\spin$ structure $\mathfrak s_{X '' }$ also induces $\mathfrak s_0 $ 
 on $K_{-p/q}$. 
 Then we obtain the $\spin$ 4-orbifold $\widehat X ''$ by deleting the plumbing $-X_{\Gamma} '$ with $\partial (-X_{\Gamma} ' ) =L ( p' , -q ' )$ from 
$X ''$ and attaching the cone $c L ( p ' , -q ' )$ over $L(p ' , -q ' )$. In this case $b_2^+ (\widehat X '' ) =0$ and $b_2^- (\widehat X  '' ) =1$ and hence 
by Corollary \ref{cor} we have 
\[
\kappa (K_{-p/q} , \mathfrak s_0 ) = \frac 18 -\overline{\mu} (L ( p ' , -q ' ) , \mathfrak s '  ) \ \text{or} \ 
2+\frac 18 -\overline{\mu} (L ( p' , -q ' ) , \mathfrak s ' ) ,
\]
where $\frac 18 -\overline{\mu} (L ( p ' , -q ' ) , \mathfrak s ' ) =-\overline{\mu} (L (p, q ) , \mathfrak s_0 ) =\kappa (L ( p, q ) , \mathfrak s_ 0 )$. 
Thus we obtain the required result in the second case. 
\end{proof}

On the other hand, the value of $\kappa$ may depend on the choice of $K$ in case of a $p/q$ surgery with $p$ and $q$ odd. In fact, 
examples below show the following proposition. 

\begin{proposition}\label{torus} 
Every integer is realized as the value of the $\kappa$-invariant of a $(+1)$ surgery along some torus knot. 
\end{proposition} 

\begin{remark}\label{rem2}
\begin{enumerate}
 \item 
 Let $L_i$ be the component of $\mathcal L$ with framing $\alpha_i$ $(L_1 =K )$, and $\mu_i$ be a meridian of $L_i$. Then $L_i$ and $\mu_i$ are 
 oriented so that we have the following relations in $H_1 (K_{p/q} , \bz )$. 
 \[
 \begin{pmatrix} -\alpha_i & -1 \\ 1 & 0 \end{pmatrix} 
\begin{pmatrix} \mu_i \\ \mu_{i-1}  \end{pmatrix}  = \begin{pmatrix} \mu_{i+1} \\ \mu_i \end{pmatrix} \ (1\le i \le n ) , 
\]
where $(\mu_1 , \mu_0 )$ corresponds to  the pair $(\mu , \lambda )$ of the meridian and the preferred longitude of $K \subset S^3$, and 
$(\mu_{n+1} , \mu_n )$ corresponds to the pair of the meridian and the longitude of the solid torus attached to 
$X_K$ by the $(p/q)$ surgery along $K$ (and hence $\mu_0 = \mu_{n+1} =0$ in $H_1 (K_{p/q} , \bz )$). 
If we put 
\[
\begin{pmatrix} \widetilde p & \widetilde q \\ \widetilde r & \widetilde s \end{pmatrix} 
=\begin{pmatrix} -\alpha_n & -1  \\ 1 & 0 \end{pmatrix}  \begin{pmatrix} -\alpha_{n-1} & -1 \\ 1 & 0 \end{pmatrix} 
\cdots \begin{pmatrix} -\alpha_1 & -1 \\ 1 & 0 \end{pmatrix}, 
\]
then $(\widetilde p , \widetilde q ) = \pm (p, q )$. We note that the set of $\spin$ structures on $K_{p/q}$ is also identified with 
the set of $\widetilde c \in \mathrm{Hom} (H_1 (S^3 \setminus \cup_{i=1}^n L_i , \bz ) , \bz_2 )$ satisfying 
\[
\sum_{j=1}^n \mathrm{lk} (L_i , L_j ) \widetilde c ( \mu_j) \equiv \mathrm{lk} (L_i , L_i ) =\alpha_i \pmod 2 \ (1 \le i \le n ).
\]
It follows that $\widetilde c \equiv 0$ (corresponding to the $\spin$ structure on $K_{p/q}$ induced by $\mathfrak s_X$), if and only if 
$c (\mu ) \equiv 0 $ (corresponding to $\mathfrak s_0$) since $\alpha_i$ is even. 
On the other hand, if $p$ is even and $q$ is odd, then $n$ is odd in the continued fraction $(\dagger )$. Hence 
$\widetilde c \in \mathrm{Hom} (H_1 (S^3 \setminus \cup_{i=1}^n L_i  , \bz ) , \bz_2 )$ satisfying 
\[
\widetilde c (\mu_i )\equiv 1\pmod 2 \quad \text{if $i$ is odd} , \ \widetilde c (\mu_ i ) \equiv 0\pmod 2  \quad \text{if $i$ is even } 
\]
corresponds to $\mathfrak s_1$, whose associated characteristic sublink of $\mathcal L$ is $\cup_{\text{$i$ odd}} L_i$.

\item $K_{1/2m}$ bounds a compact $\spin$ 4-manifold $X$ represented by a framed link consisting of $K$ and its meridian with framings $0$ and $-2m$
 respectively. A direct application of Theorem \ref{ineq} (by putting $Y_0 =S^3$) to $X$ shows that $\kappa (K_{1/2m} , \mathfrak s _0 ) = 0$ or $2$. 
 Theorem \ref{Dehn} 
slightly improves the estimate of $\kappa$ when $m>0$. 
\item If we reverse the orientation of $K_{p/q}$, we obtain $\overline{K}_{-p/q}$, where $\overline K$ is the mirror image of $K$. In case of 
$K_{p/q}$ with $p>0$, $q>0$ and $p+q \equiv 1\mod 2$, the computation of the above theorem shows that 
\[
\kappa (\overline{K}_{-p/q} , \mathfrak s_0 ) = -\kappa (K_{p/q} , \mathfrak s_0 ) \ \text{or} \  -\kappa (K_{p/q} , \mathfrak s_0 ) +2. 
\]
We can see that both cases actually 
occur (examples below). 

\item 
Let $(Y, \mathfrak s)$ be a rational homology 3-sphere $Y$ with $\spin$ structure $\mathfrak s $, which bounds a compact 4-manifold $X$ represented by a framed link 
$\mathcal L$. Suppose that 
there is a sublink $\mathcal L_0$ of $\mathcal L$ which forms a plumbing $X_{\Gamma}$ whose boundary is a spherial 3-manifold $S$, and 
contains the characteristic sublink of $\mathcal L$ corresponding to $\mathfrak s$. Then by replacing $X_{\Gamma}$ with the cone over $S$ we obtain a
 $\spin$ 4-orbifold with smaller Betti numbers, from which we could restrict the value of $\kappa (Y, \mathfrak s )$. 
 In case of $(K_{p/q}, \mathfrak s_1 )$ with $p+q \equiv  1 \pmod 2$ or $(K_{p/q} , \mathfrak s_1 )$ with 
 $p \equiv q \equiv 1 \pmod 2$, $K_{p/q}$ bounds a compact 4-manifold $X$ represented by a framed link $\mathcal L$ similar to Figure 1, which is obtained by 
 the continued fraction expansion of $p/q$. But in either case, we cannot obtain an orbifold 
 by collapsing the characteristic sublink $\mathcal L_0$ of $\mathcal L$ corresponding to 
$\mathfrak s_1$ if $K$ is nontrivial, because $\mathcal L_0$ contains a component $K$. 
 \end{enumerate}
 
\end{remark}
 
 \subsection*{Some examples} \hfill 
 
 \par
 \medskip
 Let $T_{p,q}$ be the right-handed $(p,q)$ torus knot. Then the $(1/n)$-surgery along $T_{p,q}$ is diffeomorphic to a Brieskorn homology 3-sphere up to 
 orientation as follows: 
 \begin{align*}
 &
 (T_{p,-q} )_{-1/n } \cong \Sigma (p,q,pqn-1), \ (T_{p, q })_{-1/n } \cong \Sigma (p,q,pqn+1 ), \\
 &
 (T_{p,q})_{1/n} \cong -\Sigma (p,q,pqn-1), \ (T_{p,-q } )_{1/n} \cong -\Sigma (p,q,pqn+1 ) 
 \end{align*}
 for $p,q, n >0$. Then by Theorem 5, $\kappa (-\Sigma (p,q, 2pqm-1 ) ) = \kappa (-\Sigma (p,q, 2pqm+1 ) )=0$, while 
$\kappa (\Sigma (p,q, 2pqm-1 ))$ and $\kappa (\Sigma (p,q, 2pqm+1 )$ are either $0$ or $2$ if $m>0$, and both cases occur. For example, 
$(T_{2,3} )_{-1/2} \cong \Sigma (2,3, 13)$ bounds a contractible manifold, and hence $\kappa (\Sigma (2,3,13 ))=0$. On the other hand 
$(T_{2,-3} )_{-1/2m} \cong \Sigma (2,3,12m-1 )$ and $\kappa (\Sigma (2,3, 12m-1 ) ) =2$ \cite{Mano4}. 
Another example is $(T_{2,-7})_{-1/2k} \cong \Sigma (2, 7 , 28k-1 )$. In this case $\overline{\mu} = -\beta =0$ and $\delta =2$ (\cite{Stoff}), and hence 
$\kappa (\Sigma (2,7, 28k-1 ) ) =2$ by Theorem 4. 

On the other hand, 
$\kappa (K_{1/n} )$ for $n$ odd may vary according to the choice of $K$. For example, 
$(T_{2, -4k-1 })_{1} \cong -\Sigma (2,4k+1 , 8k+3 )$, where $\Sigma (2,4k+1 , 8k+3 )$ has the Seifert invariants of the form 
\[
\{ (2,1 ) , (4k+1, 2k ) , (8k+3 , -8k-2 ) \} 
\]
(we follow the sign convention in \cite{MOY}, which is opposite to those in \cite{FFU}, \cite{U2}), and bounds a $\spin$ plumbing $X(\Gamma )$
corresponding to the weighted graph $\Gamma$ in Figure 2. 

\begin{figure}
\begin{center}
\includegraphics[width=8cm, clip]{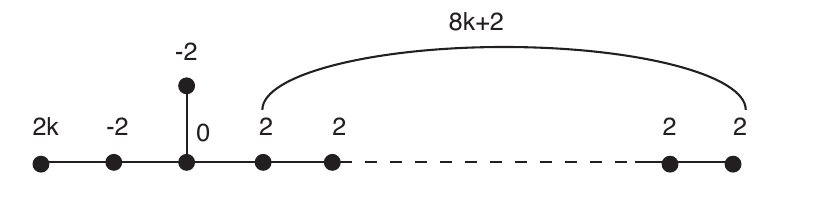}
\caption{ $\Gamma$ corresponding to $\Sigma (2,4k+1 , 8k+3 )$}
\end{center}
\end{figure}

 Hence we have $\overline{\mu} (\Sigma (2,4k+1 , 8k+3 )) =\frac 18 \sigma (X(\Gamma )) =k$.
(We note that $\sigma (X(\Gamma ))$ is the sum of the sign of $\deg (\partial X(\Gamma ))$ (which is $-1$) and the signs of the weights of the vertices of 
$\Gamma$ other than the central one in Figure 2.) 
  In \cite{U4} we claim that 
 if a Seifert rational homology 3-sphere $Y$ contains a singular fiber of even multiplicity and has positive degree, then 
 $\kappa (Y, \mathfrak s ) =-\overline{\mu} (Y, \mathfrak s )$. It follows that 
 \[
 \kappa (-\Sigma (2, 4k+1 , 8k+3 )) =-\overline{\mu} (-\Sigma (2,4k+1 , 8k+3 )) =\overline{\mu} (\Sigma (2, 4k+1 , 8k+3 )) =k . 
 \]

Likewise $(T_{2, 4k+1} )_1 \cong -\Sigma (2, 4k+1, 8k+1)$, where $\Sigma (2, 4k+1 , 8k+1 )$ has the Seifert invariants of the form 
\[
\{ (2, -1 ) , (4k+1 , -2k) , (8k+1 , 8k )\}  
\]
and hence 
\[
\kappa (-\Sigma (2, 4k+1 , 8k+1 )) = \overline{\mu} (\Sigma (2, 4k+1 , 8k+1 )) =-k  ,
\]
by computation similar to the above case. 
This proves Proposition \ref{torus}.

 \begin{figure}
\begin{center}
\includegraphics[width=8cm, clip]{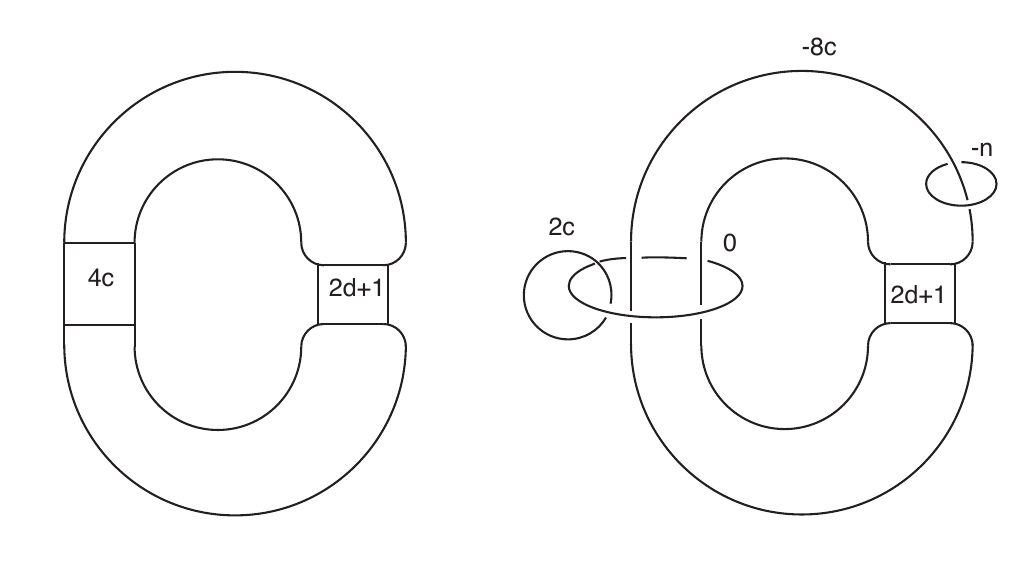}
\caption{$K(4c, 2d+1)$ and $\mathcal L$}
\end{center}
\end{figure}

 As another example, we consider the knot $K (4c, 2d+1 )$ with $c\ne 0$ in Figure 3. Here \fbox{$\ell$} means $\ell$ times right-handed half twists if $\ell >0$, and 
 $|\ell |$ times left-handed half twists if $\ell<0$. 
 
 \begin{proposition}\label{twist}
 
 Let $Y = K(4c, 2d+1 )_{1/n}$ with $c\ne 0$ and $n$ odd. Then for the unique $\spin$ structure $\mathfrak s$ on $Y$, 
we have 
\[
 \kappa (Y, \mathfrak s ) = 
 \begin{cases} -c & \text{if $n>0$}, \\
 \text{$-c$ or $-c+2$} &\text{if $n<0$.} 
 \end{cases}
 \]
 \end{proposition}
 
 We note that $\kappa (Y, \mathfrak s )=0$ if $c=0$ since $K(0, 2d+1 )$ is the unknot and hence $Y=S^3$. 
 To prove Proposition \ref{twist}, we consider a 4-manifold $X$ represented by a framed link $\mathcal L$ in 
 Figure 3, which is bounded by $Y$. We denote the components of $\mathcal L$ by $L_1$, $L_2$, $L_3$ and $L_4$ from left to right. Let $x_i$ be the 
 basis of $H_2 (X, \bz )$ represented by a union of the core of the 2-handle attached along $L_i$ and the cone over $L_i$ in the 4-ball. Then 
 the characteristic sublink of $\mathcal L$ associated with $\mathfrak s$ is $L_3$, 
 and the Poincare dual of $x_3 \mod 2$ is the 
 obstruction to extending $\mathfrak s$ to the $\spin$ structure on $X$. Then the 4-manifolds $\widetilde X$ and 
 $\widetilde X '$ represented by $L_1$ and $L_3 \cup L_4$ have 
 boundaries $L(2c, -1)$ and $L(p  , q )$ respectively, 
where $p/q = [8c, n ]$ (since $L_3$ is also the unknot). Hence by pushing $\widetilde X$ and $\widetilde X '$ separately into 
 $\mathrm{Int} X$ and replacing them with the cones over $L(2c, -1 )$ and $L(p , q  )$, we obtain a 4-orbifold $\widehat X$ with $\spin$ structure 
 $\widehat{\mathfrak s }$. Let $\widetilde{\mathfrak s}$ and $\widetilde{\mathfrak s}' $ be the $\spin$ structures on $L(2c, -1)$ and 
 $L(p, q )$ induced from 
 $\widehat{\mathfrak s }$ respectively. Then 
 \begin{align*}
& \overline{\mu} (L(2c, -1 ) , \widetilde{\mathfrak s} ) =\frac 18 (\sigma (\widetilde X )) =\frac 18 \sgn c,  \\
& \overline{\mu} (L( p, q ) , \widetilde{\mathfrak s} ' ) =\frac 18 (\sigma (\widetilde X '  ) -x_3^2 ) =\frac 18 (-\sgn [8c, n ] -\sgn n  +8c ) , 
 \end{align*}
since $\widetilde{\mathfrak s}$ extends to the $\spin$ structure on $\widetilde X $, 
while the characteristic sublink of $L_3 \cup L_4$ associated with $\widetilde{\mathfrak s ' }$ is also $L_3$. 
The intersection matrix of $X$ is congruent to the diagonal matrix whose diagonal entries are 
 $\{ 1/4n, 4/[8c, n ] , -[8c, n ] , -n \}$ and hence $\sigma (X ) =0$. It follows that 
 $b_2 (\widehat X ) =1$ and $\sigma  (\widehat X ) =\sigma (X) -\sigma (\widetilde X ) -\sigma (\widetilde X ' ) =\sgn n$ since 
 $\sgn [8c, n ]=\sgn c$. Hence $b_2^+ (\widehat X ) = 1$ and $b_2^- (\widehat X ) =0$ if $n>0$, and 
 $b_2^+ (\widehat X  ) =0$ and $b_2^- (\widehat X ) =1$ if $n<0$. 
  Thus we deduce from Corollary \ref{cor} that 
 \[
 \kappa (Y, \mathfrak s ) =-\frac{\sigma (\widehat X)}{8} - (\overline{\mu} (L (2c, -1 ), \widetilde{\mathfrak s} ) +\overline{\mu} (L(p,q ) ,\widetilde{\mathfrak s}' )) 
 =-\frac 18 (1-\sgn n  +8c) =-c
  \]
  if $n>0$, and $\kappa (Y, \mathfrak s ) =-c$ or $-c+2$ if $n<0$. 

\medskip
 We note that $K ( 4c, -1 ) =T_{2, 4c+1}$ if $c>0$, and hence if $n>0$ and $n$ is odd, 
 \[
 \kappa (K(4c, -1)_{1/n} ) =\kappa (-\Sigma (2, 4c+1 , (8c+2)n -1  ) )=\overline{\mu} (\Sigma (2,4c+1 , (8c+2)n -1 )), 
 \]
 which is $-c$ by computation similar to the above case. 
 
 \medskip
 On the other hand, if we consider $K( 2c, 2d )$ (the double twist knot with $c$ times full twists and $d$ times full twists) instead, 
 $K(2c, 2d )_{1/n}$ is diffeomorphic to the Dehn surgery along the Borromean rings with framings $-1/c$, $-1/d$, and $1/n$ respectively, and 
 also diffeomorphic to $K(2c, -2n )_{-1/d}$ and $K(2d, -2n )_{-1/c}$ by the symmetry of the Borromean rings. Hence if $c$ or $d$ is even, 
 $\kappa (K( 2c, 2d )_{1/n} )  =0$ or $2$, and if $c<0$ or $d<0$ in addition, it is $0$ by Theorem \ref{Dehn} for every nonzero integer $n$. 
 
 \medskip
 Proposition \ref{twist} can be generalized to the case of the $\pm p/q$-surgery along $K : = K (4c, 2d+1 )$ $(c\ne 0)$ as follows. 
 
 \begin{proposition}\label{twist2} 
 Let $K = K ( 4c, 2d+1 )$ with $c\ne 0$ and $\mathfrak s_1$ be the $\spin$ structure on $K_{\pm p/q}$ defined as in the first paragraph of this section. 
% Then $\kappa  (K_{p/q}  , \mathfrak s_1 )$ is described as follows. 
 
\begin{enumerate}
 \item Suppose that $p\equiv 0 \pmod 2$, $q \equiv 1 \pmod 2$, and $p>q>0$, Then there exist $\beta_i \in \bz$ $(1\le i \le n )$ such that 
 $p/q = [\beta_1 ,\cdots , \beta_n ]$, $\beta_i \equiv 0 \pmod 2$, and $|\beta_i | \ge 2$ (note that $n$ is odd in this case). 
Let 
\begin{align*} 
&
\kappa_0 = -\frac 18 (1 + \sum_{i=2}^n  \sgn \beta_i -\sum_{i \ \text{odd}} \beta_i ) -c , \\ 
&\kappa_0 ' = \frac 18 (1+\sum_{i=2}^n \sgn \beta_i -\sum_{i \ \text{odd}} \beta_i ) -c .
\end{align*}

Then 
\begin{align*}
&\kappa (K_{p/q} , \mathfrak s_1 ) = 
\begin{cases} 
\kappa_0 &\quad (\text{if $c> \frac{p}{8q} $ or $c<0$}) ,\\
\kappa_0 \ \text{or} \ \kappa_0 +2 &\quad (\text{if $0< c < \frac{p}{8q} $}) ,
\end{cases}   \\ 
&\kappa (K_{-p/q} , \mathfrak s_1 )  =
\begin{cases} 
\kappa_0 ' &\quad (\text{if $-\frac{p}{8q} < c < 0 $}),  \\
\kappa_0 ' \ \text{or} \ \kappa_0 ' +2 
&\quad (\text{if $c< -\frac{p}{8q} <0$ or $c>0$}).
\end{cases}
\end{align*} 

\item 
Suppose that $q> p > 0$, $p\equiv 0 \pmod 2$ and $q \equiv 1 \pmod 2$. Then there exist $\gamma_i \in \bz$ $(2\le i \le n )$ such that 
$q/p =[\gamma_2 , \cdots , \gamma_n ]$, $|\gamma_i | \ge 2$, $\gamma_i \equiv 0 \pmod 2$ ($n$ is also odd). 
Let 
\begin{align*} 
&\kappa_1 = -\frac 18 (1-\sum_{i=2}^n \sgn \gamma_i +\sum_{\text{$i$ odd }} \gamma_i  ) -c , \\
&\kappa_1 ' = \frac 18 (1-\sum_{i=2}^n \sgn \gamma_i +\sum_{\text{$i$ odd}} \gamma_i ) -c . 
\end{align*}
Then 
\begin{align*}
&\kappa (K_{p/q} , \mathfrak s_1 ) = \kappa_1  \\
&\kappa (K_{-p/q} , \mathfrak s_1 ) =\kappa_1 ' \ \text{or} \ \kappa_1 ' +2 
\end{align*}

\item 
Suppose that $p>0$, $q>0$, $p \equiv q \equiv 1 \pmod 2 $.  Then there exist $\delta_i \in \bz $ $(1 \le i \le n )$ such that 
$(p+q)/p =[\delta_1 , \cdots , \delta_n ]$, $|\delta_i | \ge 2$, $\delta_i \equiv 0 \pmod 2$. 
Let 
\begin{align*}
&\kappa_2 = -\frac 18 ( 1 -\sum_{i=1}^n \sgn \delta_i  ) -c ,\\
&\kappa_2 ' = \frac 18 (1-\sum_{i=1}^n \sgn \delta_i ) -c .
\end{align*} 

Then 
\begin{align*} 
&\kappa (K_{p/q} , \mathfrak s_1 ) = 
\begin{cases} 
\kappa_2 &\quad (\text{if $c<0$ or $\frac{p}{8q} < c$}) ,\\
\kappa_2 \ \text{or} \ \kappa_2 +2 &\quad (\text{if $0< c < \frac{p}{8q} $} ) ,
\end{cases}  \\
&\kappa(K_{-p/q} , \mathfrak s_1 ) = 
\begin{cases} 
\kappa_2 ' &\quad (\text{if -$\frac{p}{8q} <c <0$ } ), \\
\kappa_2 ' \ \text{or} \ \kappa_2 ' +2 
&\quad (\text{if $c>0$ or $c < -\frac{p}{8q}$} ) .
\end{cases} 
\end{align*}

\end{enumerate}
\end{proposition}
 
 \begin{proof} 
 Suppose that $p>q>0$, $p\equiv 0 \pmod 2$, and $q\equiv 1 \pmod 2$. Let $\mathcal L_{(1)}$ be a framed link obtained from 
 $\mathcal L$ in the proof of Proposition \ref{twist} by replacing $L_4$ with $\cup_{i=1}^n  L_i '$, where 
 \begin{itemize}
 \item 
 $L_1 ' =L_4$ and $L_i'$ is a meridian of $L_{i-1} '$ $(i\ge 2 )$, 
 \item 
 the framings of $L_i'$ are $-8c+\beta_1$ $(i=1)$, and $\beta_i$ $(i\ge 2 )$ respectively. 
 \end{itemize}
 Then $K_{p/q}$ bounds a 4 manifold $X_{(1)}$ represented by $\mathcal L_{(1)}$ and the characteristic sublink of $\mathcal L_{(1)}$ associated with 
 $\mathfrak s_1$ is $\cup_{\text{$i$ odd}} L_i '$. Furthermore $\cup_{i=1}^n  L_i '$ represents a plumbing $X_{\Gamma_{(1)}}$ with 
 $\partial X_{\Gamma_{(1)}} = L(p-8cq, -q )$. 
 The intersection matrix of $X_{(1)}$ is congruent to the diagonal matrix whose diagonal entries are 
 \[
 \{ p/(4q ), -4q/(p-8cq) , (p-8cq)/q, [\beta_i , \cdots , \beta_n ] \ (i\ge 2 ) \}. 
\]
Furthermore the intersection matrices of the 4-manifold $X_{L_1}$ corresponding to $L_1$ and $X_{\Gamma_{(1)}}$ are 
 congruent to the diagonal matrices whose diagonal entries are $\{ 2c \}$ and 
 $\{ (p-8cq)/q , [\beta_i , \cdots , \beta_n ] \ (i\ge 2 ) \} $ respectively. 
 Then by pushing $X_{\Gamma_{(1)}}$ and $X_{L_1}$ into $X_{(1)}$ and 
 replacing them with the cones over $L(p-8cq , -q )$ and $L(2c , -1 )$ we obtain a $\spin$ orbifold $\widehat  X_{(1)}$ bounded by $K_{p/q}$ whose 
 $\spin$ structure $\widehat{\mathfrak s}$ induces $\mathfrak s_1$ on $K_{p/q}$. Let 
 $\mathfrak s_{(1)}$ and $\mathfrak s_{L_1}$ be the $\spin$ structures on $L(p-8cq , q )$ and $L(2c , -1 )$ induced by $\widehat{\mathfrak s}$ respectively. 
 It follows that $\sigma (\widehat X_{(1)} ) =\sigma (X_{(1)}) -\sigma (X_{\Gamma_{(1)}} ) -\sigma (X_{L_1} ) = 1 -\sgn (p-8cq )   -\sgn c$ and hence 
 \begin{align*}
& b^+_2 (\widehat X_{(1)} ) = 1, \ b^-_2 (\widehat X_{(1)}) =0 \ \text{if $c<0$ or $p/q  < 8c$},  \\
& b^+_2 (\widehat X_{(1)} ) =0, \ b^-_2 (\widehat X_{(1)}) =1 \ \text{if $0<8c < p/ q$}. 
\end{align*}
 Furthermore we have 
 \[
  \overline{\mu} (L(p-8cq , q ) , \mathfrak s_{(1)} ) = \frac 18 (\sgn (p-8cq) +\sum_{i=2}^n  \sgn \beta_i -\sum_{\text{$i$ odd}} \beta_i  ) +c, 
  \]
  and $\overline{\mu } (L(2c , -1 ) = \frac 18 \sgn c$  since $\sgn [\beta_i , \cdots , \beta_n ] = \sgn \beta_i$. 
 Let $\mathcal L_{(1)} '$ be a framed link obtained from $\mathcal L_{(1)}$ by replacing the framings of $L_i '$ with 
 $-8c -\beta_1$ $(i=1)$ and $-\beta_i$ $(i \ge  2 )$ respectively. Then $\mathcal L_{(1)} '$ represents a 4-manifold $X_{(1)} '$ bounded by $K_{-p/q}$ and the characteristic sublink of $\mathcal L_{(1)} ' $ is 
 also $\cup_{\text{$i$ odd}} L_i '$. 
In this case $\cup_{i=1}^n  L_i '$ represents a plumbing $X_{\Gamma_{(1)}'}  $ with $\partial X_{\Gamma_{(1)} '} = 
L( p+8cq , q ) $. 
We obtain a $\spin$ orbifold $\widehat X_{(1)} '$ with $\partial \widehat X_{(1)} '  =K_{-p/q}$ by replacing $X_{\Gamma_{(1)} '} $ and $X_{L_1}$ with 
the cones over their boundaries, whose $\spin$ structure $\widehat{\mathfrak s '}$ induces $s_1$ on $K_{-p/q}$. 
By computation similar to the above case, we have $\sigma (\widehat X_{(1)} ' ) = 
\sgn (p+8cq)-1 -\sgn c$ and hence 
\begin{align*}
&b_2^+ (\widehat X_{(1)} ' ) =1, \ b_2^- (\widehat X_{(1)} ' ) =0 \ \text{if $-p/q < 8c <0$} , \\
&b_2^- (\widehat X_{(1)} ' ) =0, \ b_2^- (\widehat X_{(1)} ' ) =1 \ \text{if $c>0$ or $8c < -p/q $}. 
\end{align*}
Furthermore we have 
\[
\overline{\mu} (L( p+8cq , q ) , \mathfrak s_{(1) '} ) = \frac 18 (-\sgn (p+8cq) -\sum_{i=2}^n \sgn \beta_i +\sum_{\text{$i$ odd}} \beta_i  )  +c , 
\]
where $\mathfrak s_{(1) '}$ is the $\spin$ structure induced from $\widehat {\mathfrak s '}$. 
 
 \medskip
 Next suppose that $q>p>0$, $p\equiv 0 \pmod 2$ and $q\equiv 1 \pmod 2$. 
 Let $\mathcal L_{(2)}$ be a framed link obtained from $\mathcal L$ by replacing $L_4$ with 
 $\cup_{i=1}^n  L_i '$, where 
 \begin{itemize}
 \item 
 $L_1 ' = L_4$, $L_i '$ is a meridian of $L_{i-1} '$ $(i\ge 2 )$, 
 \item 
 the framings of $L_i '$ are $-8c$ $(i=1)$, $-\gamma_i$ $(i\ge 2 )$ respectively. 
 \end{itemize}
 Then $\mathcal L_{(2)}$ represents a 4-manifold $X_{(2)}$ bounded by $K_{p/q}$ and the characteristic sublink of $\mathcal L_{(2)}$
associated with $\mathfrak s_1$ is 
 $\cup_{\text{$i$ odd}} L_i '$. Furthermore $\cup_{i=1}^n L_i '$ represents a plumbing $X_{\Gamma_{(2)}}$ with 
 $\partial X_{\Gamma_{(2)}} = L( p-8cq , -q )$. We can also form a $\spin$ orbifold $\widehat X_{(2)}$ from $X_{(2)}$ by replacing 
 $X_{\Gamma_{(2)} }$ and $X_{L_1}$ with the cones over their boundaries, whose $\spin$ structure $\widehat{\mathfrak s}_{(2)}$ induces 
 $\mathfrak s_1$ on $K_{p/q}$. The intersection matrix of $X_{(2)}$ is congruent to the diagonal matrix whose diagonal entries are 
 \[
 \{ p/(4q) , -4q/ (p-8cq ), (p-8cq)/ q , -[\gamma_i , \cdots , \gamma_n ]  \ (i\ge 2 ) \} 
 \]
 and we have $\sigma (\widehat X_{(2)} ) = 1-\sgn (p-8cq ) -\sgn c$ by computation similar to the above case. Hence 
we have $b_2^+ (\widehat X_{(2)} ) = 1 , \ b_2^- (\widehat X_{(2)}) = 0$ since $0<p/q <1$ 
and 
\[
\overline{\mu} (L((p-8cq , -q ) , \mathfrak s_{(2)}  ) =\frac 18 (\sgn (p-8cq)   - \sum_{i=2}^n  \sgn \gamma_i +\sum_{\text{$i$ odd}} \gamma_i ) 
+c 
\]
and $\overline{\mu} (L(2c , -1 ) , \mathfrak s_{L_1} ) = \sgn c$, where $\mathfrak s_{(2)}$ and $s_{L_1}$ are the $\spin$ structures induced by 
$\widehat{\mathfrak s}_{(2)}$. 
 On the other hand, $K_{-p/q}$ bounds a 4-manifold $X_{(2)} '$ represented by a framed link $\mathcal L_{(2)} '$, which is obtained from 
 $\mathcal L_{(2)}$ by replacing the framings of $L_i '$ with $\gamma_i$ $(i\ge 2 )$. The characteristic sublink of $\mathcal L_{(2)} '$ 
 associated with $\mathfrak s_1$ is also $\cup_{\text{$i$ odd}} L_i '$ and $\cup_{i=1}^n L_i '$ forms a plumbing $X_{\Gamma_{(2)} '}$ 
 with $\partial X_{\Gamma_{(2)} '} =L(p+8cq , q )$. We also form a $\spin$ 4-orbifold $\widehat X_{(2)} '$ from $X_{(2)} '$ by 
 replacing $X_{\Gamma_{(2)} '}$ and $X_{L_1}$ with the cones over their boundaries, whose $\spin$ structure $\widehat{\mathfrak s}_{(2) '}$ induces 
 $\mathfrak s_1$ on $K_{-p/q}$. Then computation similar to the above case shows that 
 $\sigma (\widehat X_{(2)} ' ) =\sgn (p+8cq)  -1 -\sgn c$, and $b_2^+ (\widehat X_{(2)} ' ) =0, \ b_2^- (\widehat X_{(2)} )=1$.
 Furthermore we have 
\[
\overline{\mu} (L(p+8cq , q  ) , \mathfrak s_{(2)} '  ) = \frac 18 (- \sgn (p+8cq) +\sum_{i=2}^n \sgn \gamma_i  -\sum_{\text{$i$ odd}} \gamma_i )  +c 
 \]
 where $\mathfrak s_{(2)} '$ is induced from $\widehat{\mathfrak s_{(2)}} '$.

 \medskip
 Finally suppose that  $p>0$, $q>0$, and $p\equiv q \equiv 1 \pmod 2$. In this case we consider a framed link 
 $\mathcal L_{(2)}$ obtained from $\mathcal L$ by replacing $L_4$ with 
 $\cup_{i=1}^{n+2} L_i '$, where 
 \begin{itemize}
 \item 
 $L_1 ' =L_4$, $L_i ' $ is a meridian of $L_{i-1} '$ $(i\ge 2 )$, 
 \item 
 the framings of $L_i '$ are $-8c-1$ $(i=1)$, $-1$ $(i=2 )$, $-\delta_{i-2}$ $(i\ge 3 )$ respectively. 
 \end{itemize}
 Then $\mathcal L_{(2)}$ represents a 4-manifold $X_{(2 )}$ bounded by $K_{p/q}$, and $\cup_{i=1}^ {n+2} L_i '$ represents a plumbing 
 $X_{\Gamma_{(2)}}$ bounded by $L(p-8cq , -q )$. The characteristic sublink of $\mathcal L_{(2)}$ associated with $\mathfrak s_1$ is 
 $L_1 '$. The intersection matrices of $X_{(2)}$ and $X_{\Gamma_{(2)}}$ are congruent to the diagonal matrices whose diagonal entries are 
 \begin{align*}
 &\{ p/(4q), -4q/ (p-8cq) , (p-8cq)/q , -q/ (p+q), -[\delta_i , \cdots , \delta_n ] \  (i \ge 1 ) \} \\
 &\{  (p-8cq)/q , -q/(p+q) , -[\delta_i , \cdots , \delta_n ] \ (i\ge 1 ) \}.  
 \end{align*}
It follows that  if we form a $\spin$ 4-orbifold $\widehat X_{(2)}$ from $X_{(2)}$ by replacing $X_{\Gamma_{(2)}}$ and $X_{L_1}$ with the cones over their 
 boundaries, whose $\spin$ structure $\widehat{\mathfrak s}_{(2)}$ induces $\mathfrak s_1$ on $K_{p/q}$, we have 
 $\sigma (\widehat X_{(2)} ) = 1 -\sgn (p-8cq)  -\sgn c$ and hence 
 \begin{align*}
  & b_2^+ (\widehat X_{(2)} ) =1, \ b_2^- (\widehat X_{(2)} ) =0 \ \text{if $c<0$ or $8c>p/q$} , \\
  &b_2^+ (\widehat X_{(2)}) =0 , \ b_2^- (\widehat X_{(2)} ) =1 \ \text{if $0<8c <p/q$}. 
  \end{align*} 
  Furthermore we have 
  \[
  \overline{\mu} (L ( p-8cq , -q ) , \mathfrak s_{(2)} ) = \frac 18 (\sgn (p-8cq ) -\sum_{i=1} \sgn \delta_i  ) +c ,
  \]
  where $\mathfrak s_{(2)}$ is induced by $\widehat{\mathfrak s}_{(2)}$. 
  On the other hand $K_{-p/q}$ bounds a 4-manifold $X_{(2)} '$ represented by a framed link $\mathcal L_{(2)} '$, which is obtained from 
  $\mathcal L_{(2)}$ by replacing the framings of $L_i '$ with 
  $ -8c+1$ $(i=1)$, $+1$ $(i=2 )$, $\delta_{i-2}$ $(i\ge 3 )$ respectively. In this case 
  $\cup_{i=1}^{n+2} L_i ' $ forms a plumbing $X_{\Gamma_{(2)} '}$ bounded by $L(p+8cq , q )$, and the characteristic sublink of 
  $\mathcal L_{(2)} '$ is also $L_1 '$. If we form a $\spin$ orbifold $\widehat X_{(2)} '$ from $X_{(2) } '$ by replacing 
  $X_{\Gamma_{(2)} ' }$ and $X_{L_1}$ with the cones over their boundaries, whose $\spin$ structure $\widehat{\mathfrak s}_{(2)} '$ induces 
  $\mathfrak s_1$ on $K_{-p/q}$ as before, computation similar to the above case shows that 
  $\sigma (\widehat X_{(2)} ' ) =\sgn (p+8cq ) -1 -\sgn c$, 
  \begin{align*}
  &b_2^+ (\widehat X_{(2)} ' ) =1, \ b_2^- (\widehat X_{(2)} ' ) =0 \ \text{if $-p/q <8c<0$}, \\
  &b_2^+ (\widehat X_{(2)} ' ) =0, \ b_2^- (\widehat X_{(2)} ' ) =1 \ \text{if $c>0$ or $8c<-p/q$ },
  \end{align*} 
  and also 
  \[
  \overline{\mu} (L( p+8cq, q ) , \mathfrak s_{(2)} ' ) =\frac 18 (-\sgn (p+8cq) +\sum_{i=1}^n \sgn \delta_i  ) +c ,
  \]
  where $\mathfrak s_{(2)} '$ is induced from $\widehat{\mathfrak s}_{(2) '}$. 
  \medskip
  
  We obtain the required results by applying Corollary \ref{cor} to the above computation. 
  
\end{proof}


\begin{thebibliography}{99}


\bibitem{Dai} I. Dai, On the $\Pin (2)$-equivariant monopole Floer homology of plumbed 3-manifolds, Mich. Math. J. {\textbf 67} (2018), 423--447. 

\bibitem{DM} I. Dai, C. Manolescu, Involutive Heegaard Floer homology and plumbed three-manifolds, J. Inst. Math. Jussieu {\textbf 18} (2019), 1115--1155. 

\bibitem{DSS} I. Dai, H. Sasahira, M. Stoffregen, Lattice homology and Seiberg-Witten-Floer spectra, arXiv:2309.01253, (2023) 

\bibitem{FF} Y. Fukumoto, M. Furuta, 
Homology 3-spheres bounding acyclic 4-manifolds, 
Math. Res. Lett. {\textbf 7} (2000), 757--766 

\bibitem{FFU}  Y. Fukumoto, M. Furuta, M. Ue, 
$W$ invariants and the Neumann-Siebenmann 
invariants for Seifert homology 3-spheres,  Topol. and its appl.
{\textbf {116}}  (2001), 
333--369

\bibitem{F} K. A. Fr{\o}yshov, Monopole Floer homology for rational homology 3-manifolds, Duke Math. J. {\text{155}} (2010), 519--576
\bibitem{HM} K. Hendricks, C. Manolescu, Involutive Heegaard Floer Homology, Duke  Math. J. {\textbf 166 (7)} (2017), 1211--1299. 

\bibitem{KM} P. Kronheimer, T. Mrowka, Monopoles and Three-Manifolds, New Mathematical Monograph {\textbf{10}} Cambridge Univ. Press, (2007) 

\bibitem{MOY} T. Mrowka, P. Ozsv\'ath, and B. Yu, Seiberg-Witten monopoles on Seifert 
fibered spaces, Comm. Anal. Geom., {\textbf{5}} (1997), No. 4, 685--791

\bibitem{Mano1} C. Manolescu, Seiberg-Witten Floer stable homotopy type of three manifolds with $b_1 =0$, Geom. Topol. {\textbf 7} (2003), 889--932.

%\bibitem{Mano2} C. Manolescu, A gluing theorem for the relative Baure-Furuta invariants, J. DIff. Geom. {\textbf 76} (2007), 117--153. 

\bibitem{Mano3} C. Manolescu, $\Pin (2)$-equivariant Seiberg-Witten Floer homology and the 
triangulation conjecture, J. Amer. Math. Soc. {\textbf 29} (2016), 147--176. 

\bibitem{Mano4} C. Manolescu, On the intersection forms of spin 4-manifolds with boundary, Math. Ann. {\textbf 359} (2014), 695--728.
\bibitem{N}  W. D. Neumann, An invariant of plumbed homology spheres, in Topology Symposium, 
Siegen 1979 (Proc. Sympos., Univ. Siegen), Springer, Berlin Vol. {\textbf{788}} Lectures 
Notes in Math., pp.125--144, (1980) 

%\bibitem{NR} W. D. Neumann, F. Raymond, Seifert manifolds, plumbing, $\mu$-invariant and 
%orientation reversing maps, in Algebraic and geometric topology (Proc. Sympos., Univ. California, 
%Santa Barbara, 1977), Springer, Berlin, Vol. {\textbf{664}} Lecture Notes in math., pp. 
%163--196, (1978) 

\bibitem{Ne1} A. N\'emethi, On the Ozsv\'ath-Szab\'o invariant of negative definite plumbed 3-manifolds, Geom. Topol. {\textbf{9}} (2005), 991--1042

\bibitem{OS} P. Ozsv\'ath, Z. Szab\'o, Absolutely graded 
Floer homologies and intersection forms for four manifolds, 
Adv. in Math., {\textbf {173}} (2003), 179--261

\bibitem{S}  N. Saveliev,  Fukumoto-Furuta invariants of plumbed homology 
3-spheres, Pacific J. Math. {\textbf {205}} (2002), 465--490

\bibitem{Si} L. Siebenmann, On vanishing of the Rohlin invariant and nonfinitely 
amphicheiral homology 3-spheres, in 
Topology Symposium, Siegen 1979 (proc. Sympos., Univ. Siegen, 1979) , 
Springer, Berlin, Vol. {\textbf{788}} lecture Notes in Math., pp. 172--222, (1980)

\bibitem{Stoff} M. Stoffregen, $\Pin (2)$-equivariant Seiberg-Witten Floer homology of Seifert fibrations, Compos. Math. {\textbf 156 (2)} (2020), 199-250. 

%\bibitem{U1}  M. Ue,  On the intersection forms of spin 4-manifolds 
%bounded by spherical 3-manifolds, 
%Alg. Geom. Topol. {\textbf 1} (2001),  549--578

\bibitem{U2}  M. Ue,  The Neumann-Siebenmann invariant and Seifert surgery, 
Math.Z. {\textbf {250}} (2005), 475--493

%\bibitem{U3} M. Ue,  The Fukumoto-Furuta and the Ozsv\'ath-Szab\'o invariants 
%for spherical 3-manifolds, Algebraic Topology -old and new, Banach Center Publ. 
%{\textbf {85}} (2009), 121--139

\bibitem{U4} M. Ue, Constraints on intersection forms of spin 4-manifolds bounded by Seifert rational homology 3-spheres in terms of $\overline{\mu}$ and $\kappa$ invariants, 
arXiv:2206.05412, (2022) 

\end{thebibliography}
\end{document}